\documentclass{article}
\usepackage{amssymb,amsmath}
\usepackage{graphicx}

\def\qed{\hfill {\hbox{${\vcenter{\vbox{               
   \hrule height 0.4pt\hbox{\vrule width 0.4pt height 6pt
   \kern5pt\vrule width 0.4pt}\hrule height 0.4pt}}}$}}}

\def\tr{\triangleright}

\def\mb{\mathbf}

\newtheorem{theorem}{Theorem}
\newtheorem{definition}{Definition}
\newtheorem{lemma}[theorem]{Lemma}
\newtheorem{proposition}[theorem]{Proposition}
\newtheorem{corollary}[theorem]{Corollary}
\newtheorem{example}{Example}

\newtheorem{remark}{Remark}

\newenvironment{proof}[1][Proof]{\smallskip\noindent{\bf #1.}\quad}%
{\qed\par\medskip}

\date{}

\title{\Large \textbf{$(t,s)$-racks and their link invariants}}

\author{Jessica Ceniceros\footnote{Email: \texttt{jceniceros11@students.claremontmckenna.edu}}
 \and Sam Nelson\footnote{Email: \texttt{knots@esotericka.org}}}

\begin{document}
\maketitle

\begin{abstract}
A \textit{$(t,s)$-rack} is a rack structure defined on a module over the 
ring $\ddot\Lambda=\mathbb{Z}[t^{\pm 1},s]/(s^2-(1-t)s)$. We identify necessary 
and sufficient conditions for two $(t,s)$-racks to be isomorphic.
We define enhancements of the rack counting invariant using the 
structure of $(t,s)$-racks and give some computations and examples.
As an application, we use these enhanced invariants to obtain 
obstructions to knot ordering.
\end{abstract}

\begin{center}
\parbox{6in}{
\textsc{Keywords:} Finite racks, $(t,s)$-racks, Alexander quandles,
link invariants, enhancements of counting invariants

\textsc{2010 MSC:} 57M27, 57M25
}
\end{center}

\section{\large\textbf{Introduction}}

Introduced in \cite{FR}, the \textit{fundamental rack} of a framed 
link is a complete invariant of unsplit framed links in $S^3$ up to 
homeomorphism of $S^3$. Counting homomorphisms from a fundamental 
rack into a finite rack $X$ yields an invariant of framed isotopy. In \cite{N} 
it is shown that this counting invariant is periodic with respect to 
framings modulo an integer $N(X)$ known as the \textit{rack rank} of $X$, 
and that summing these counting invariants over a complete period of framings
module $N(X)$ yields an invariant of ambient isotopy.

In this paper we study a type of rack structure on modules over the ring
$\ddot\Lambda=\mathbb{Z}[t^{\pm 1},s]/(s^2-(1-t)s)$ known as 
\textit{$(t,s)$-racks} and the counting invariants they define. We obtain a
result specifying necessary and sufficient conditions for two $(t,s)$-racks 
to be isomorphic, similar to results for Alexander quandles and Alexander 
biquandles in \cite{N1} and \cite{LN} respectively.
We are able to exploit the module structure of these racks to enhance the 
rack counting invariant, yielding stronger invariants which 
specialize to the unenhanced counting invariant. 

The paper is organized as follows. In section \ref{sec2} we review the basics 
of racks and the rack counting invariant. In section \ref{sec3} we introduce
$(t,s)$-racks and provide necessary and sufficient conditions for two 
$(t,s)$-racks to be isomorphic. In section \ref{sec4} we define the new 
enhanced invariants, give examples and provide an application to knot ordering. 
In section \ref{sec5} we  collect questions for future research.

\section{\large\textbf{Rack basics}} \label{sec2}

\begin{definition}
\textup{A \textit{rack} is a set $X$ with two binary operations $\tr$,
$\tr^{-1}$ satisfying for all $x,y\in X$
\begin{list}{}{}
\item[(i)]{$(x\tr y)\tr^{-1}y = x=(x\tr^{-1} y)\tr y$ and}
\item[(ii)]{$(x\tr y)\tr z=(x\tr z)\tr(y\tr z)$.}
\end{list}
It follows from (i) and (ii) (see \cite{N}) that the \textit{kink map} 
$\pi:X\to X$ defined 
by $\pi(x)=x\tr x$ is a bijection. For every element $x\in X$, the 
\textit{rack rank} of $x$ is the smallest integer $N(x)\ge 1$ such that 
$\pi^{N(x)}(x)=x$ or $\infty$ if no such $N(x)$ exists, and the 
\textit{rack rank} of $X$ is smallest positive integer $N(X)$ such that 
$\pi^{N(X)}(x)=x$ for all $x\in X$, or $\infty$ if no such $N$ exists. 
A rack with rack rank $N=1$ is a \textit{quandle}. We will denote the
kink map in $X$ by $\pi_X:X\to X$ when necessary to distinguish it from
the kink maps of other racks.}
\end{definition}

We have the following standard result (or see \cite{N}):

\begin{lemma}
If $X$ is a finite rack, then $N(x)\ne \infty$ for all $x\in X$ and 
$N(X)=\mathrm{lcm}\{N(x)\ |\ x\in X\}$.
\end{lemma}

\begin{proof}
Let $X$ be a finite rack and condsider the map $f_x:X\to X$ defined by 
$f_x(y)=y\tr x$. For each $x$, $f_x$ is a element of the symmetric group 
$S_{|X|}$ and hence has finite order equal to $N(x)$. Since $N(x)$ must divide
$N(X)$, for all $x\in X$, we must have $N(X)=\mathrm{lcm}\{N(x)\ |\ x\in X\}$.
\end{proof}

As with other algebraic structures, we have some useful standard concepts:
\begin{definition}
\textup{Let $X$ and $Y$ be racks.}
\begin{list}{$\bullet$}{}
\item \textup{A \textit{subrack} of $X$ is a subset $S\subset X$ which is 
itself a rack under the rack operations $\tr,\tr^{-1}$ inherited from $X$. For
$S\subset X$ to be a subrack, it is sufficient for $S$ to be closed under the 
rack operations $\tr$ and $\tr^{-1}$. If the rack rank of $X$ is finite, 
then closure under $\tr$ implies closure under $\tr^{-1}$.}
\item \textup{A \textit{rack homomorphism} is a map $f:X\to Y$ satisfying 
for all $x,y\in X$} \[f(x\tr y)=f(x)\tr f(y)\quad  \mathrm{and} \quad
f(x\tr^{-1} y)=f(x)\tr^{-1} f(y).\]
\item \textup{The \textit{image} of a homomorphism $f:X\to Y$ is the set
$\mathrm{Im}(f)=\{y\in Y\ | \ y=f(x) \ \mathrm{for\ some} \ x\in X \}$; it 
is straightforward to show that $\mathrm{Im}(f)$ is a subrack of $Y$.}
\end{list}
\end{definition}

We will find the following observations useful in section \ref{sec2}.

\begin{lemma} \label{lem:qord}
Let $f:X\to Y$ be a rack homomorphism. Then for any $x\in X$, the
rack rank of $f(x)$ divides the rack rank of $x$.
\end{lemma}

\begin{proof} Let $f:X\to Y$ be a rack homomorphism. Then for any $x\in X$
we have 
\[f(\pi_X(x))=f(x\tr x)=f(x)\tr f(x)=\pi_Y(f(x)).\]
Then if $\pi_X^N(x)=x$, we have $\pi_Y^N(f(x))=f(\pi_X^N(x))=f(x)$ so
$\pi_Y^{N(X)}(f(x))=f(x)$, and $N(f(x))|N(x)$ as required.
\end{proof}

\begin{corollary}
If $f:X\to Y$ is an isomorphism of racks then $\pi_X=f^{-1}\pi_Yf$, i.e.
the kink maps of $X$ and $Y$ are conjugate. 
\end{corollary}

\begin{corollary}
If two racks $X$ and $Y$ are isomorphic, then the rack ranks of $X$ and $Y$ are
equal.
\end{corollary}

The rack axioms come from the blackboard-framed oriented 
Reidemeister moves where we interpret $x\tr y$ as the arc resulting
from $x$ crossing under $y$ from right to left with respect to the
orientation of the overcrossing strand and $x\tr^{-1} y$ as crossing
under from left to right \cite{FR}.
\[\includegraphics{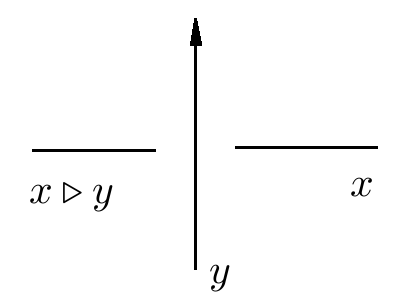}\quad 
\includegraphics{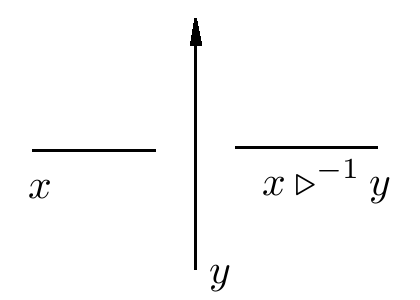}\]

Axiom (i) comes from Reidemeister move II, axiom (ii) comes from the
oriented Reidemeister III move with all positive crossings. 

\[\includegraphics{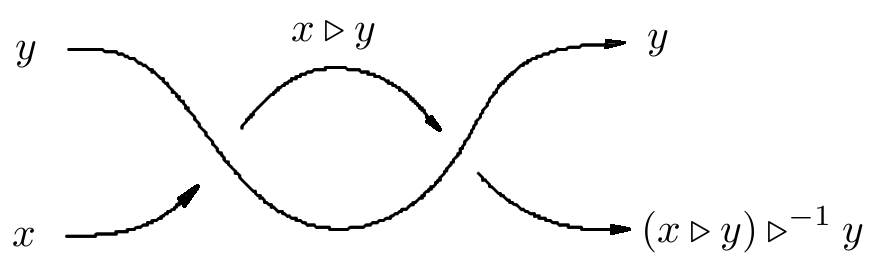}\quad 
\includegraphics{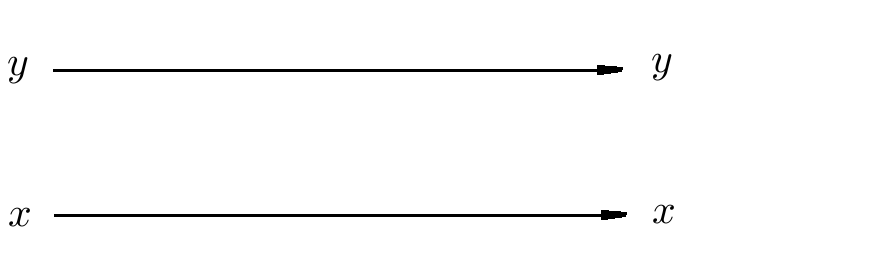}\]

\[\includegraphics{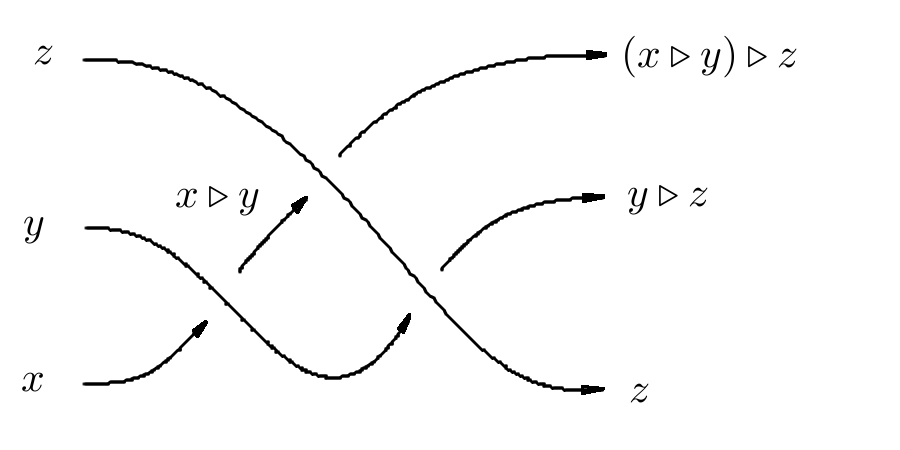}\quad 
\includegraphics{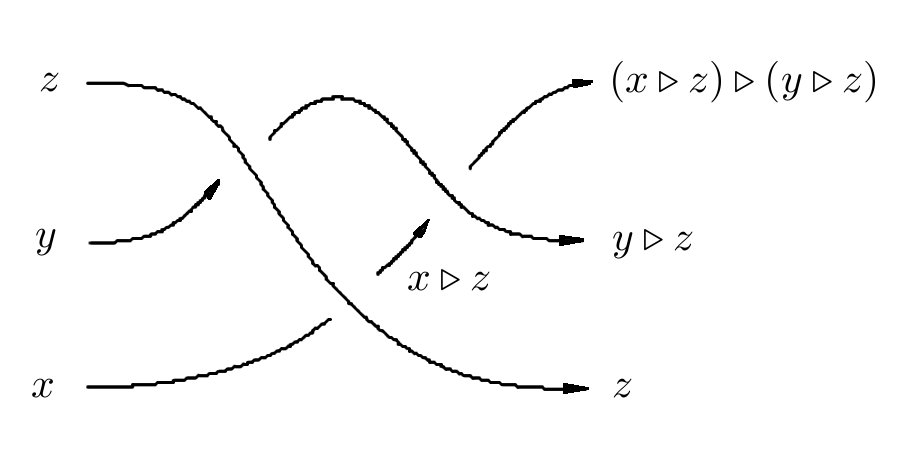}\]

The other oriented Reidemeister III moves follow from the listed moves, 
with corresponding rack equations such as 
\[(x\tr y)\tr^{-1} z=(x\tr^{-1}z)\tr(y\tr^{-1}z).\]
See \cite{FR} for more.

The blackboard-framed oriented Reidemeister I moves do not impose any
additional axioms, but provide a visual interpretation of the kink map:
$\pi(x)$ is the result of $x$ going through a positive-writhe kink,
and $\pi^{-1}(x)$ is the result of going through a negative-writhe
kink.

\[\includegraphics{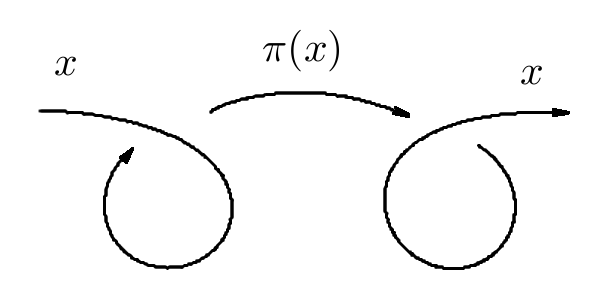}\quad 
\includegraphics{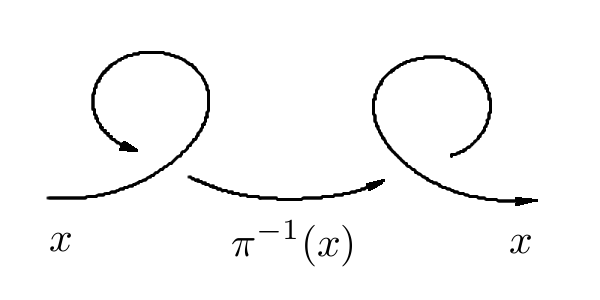}\]

Standard examples of rack structures include:
\begin{list}{$\bullet$}{}
\item{\textit{Constant action racks:} A set $X$ with a bijection 
$\sigma:X\to X$ is a rack with $x\tr y=\sigma(x)$,}
\item{\textit{Conjugation racks:} A group $G$ is a rack with 
$x\tr y=y^{-n}xy^n$ for each $n\in \mathbb{Z}$,}
\item{\textit{Coxeter racks:} The subset $S\subset V$ of an 
$\mathbb{F}$-vector space $V$ which is non-degenerate with respect to 
a symmetric bilinear form $\langle,\rangle:V\times V\to \mathbb{F}$ is a
rack with
\[\mb{x}\tr\mb{y}=\mb{y}-2\frac{\langle \mb{x},\mb{y} 
\rangle}{\langle \mb{x},\mb{x} \rangle}\mb{x},\]}
\item{\textit{Fundamental rack of a link $L$:} Let $L$ be a blackboard-framed
oriented link diagram and let $G$ be a set of generators corresponding 
bijectively with the set of arcs in $L$. Define the set $W(G)$ of 
\textit{rack words in $G$} recursively by the rules
\begin{itemize}
\item $g\in G\ \Rightarrow\ g\in W(G)$ and
\item $g,h\in W(G)\ \Rightarrow\ g\tr h\in W(G)\ \mathrm{and} 
\ g\tr^{-1} h\in W(G)$.
\end{itemize}
Let $\sim$ be the equivalence relation on $W(G)$ generated by the the
rack axioms (e.g., $(x\tr y)\tr z\sim (x\tr z)\tr(y\tr z)$ etc.) together 
with the crossing relations in $L$, i.e., for every crossing  
\[\includegraphics{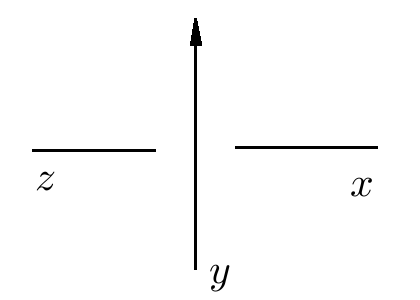}\]
we obtain a relation $z\sim x\tr y$ or, equivalently, $x\sim z\tr^{-1} y$.
Then the set $FR(L)=W(G)/\sim$ of equivalence classes in $W(G)$  modulo the 
equivalence relation $\sim$ is a rack under the operations 
\[[x]\tr [y]=[x\tr y] \quad \mathrm{and} \quad
[x]\tr^{-1} [y]=[x\tr^{-1} y]\]
where $[x]$ is the equivalence class of $x$ in $W(G)/\sim$. In particular,
the fundamental racks of any two oriented blackboard-framed link diagrams
related by blackboard framed Reidemeister moves are isomorphic.}
\end{list}

This last example is especially important; in \cite{FR} it is shown that the
Fundamental Rack of a framed link is a complete invariant for unsplit framed
links, up to homeomorphism of the ambient 
space $S^3$. For example, the blackboard-framed trefoil below
has fundamental rack with the listed presentation: 
\[\includegraphics{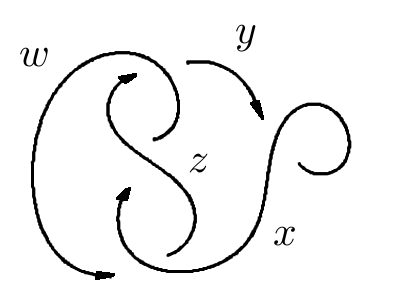} \quad \raisebox{0.5in}{$FR(K)=
\langle x,y,z,w  \ |\  y\tr x=x, z\tr x=w, w\tr z=x, y\tr w=z \rangle.$}\]

A finite rack $X=\{x_1,\dots, x_n\}$ can be expressed in an algebra-agnostic
way using a \textit{rack matrix} which encodes the operation table of $X$.
Specifically, the entry in row $i$ column $j$ of $M_X$ is $k$ where 
$x_i\tr x_j=x_k$. For example, the constant action rack on $X=\{x_1,x_2,x_3\}$
with $\sigma=(123)$ has rack matrix
\[M_X=\left[\begin{array}{ccc}
2 & 2 & 2 \\
3 & 3 & 3 \\
1 & 1 & 1 \\
\end{array}\right].\]
The kink map $\pi(x)$ is the permutation along the diagonal of the rack 
matrix; in this example, $\pi=(123)$ and $N=3$.

By construction, any labeling of a diagram $D$ of a blackboard-framed oriented 
link $L$ with elements of a rack $X$ satisfying the crossing condition at 
every crossing corresponds to a unique such labeling on any diagram obtained 
from $D$ by a blackboard-framed Reidemeister move. More abstractly, such a 
labeling is an assignment of an image $f(g)=x_i\in X$ to each generator $g$ 
of $FR(L)$, and satisfaction of the crossing conditions says that $f$  
defines a unique homomorphism of racks $f:FR(L)\to X$. The set of such 
labelings or homomorphisms is an invariant of blackboard framed isotopy 
denoted 
\[\mathrm{Hom}(FR(L),X)=\{f:FR(L)\to X\ |\ f(x\tr y)=f(x)\tr f(y)\}.\]
The cardinality $|\mathrm{Hom}(FR(L),X)|$ is a numerical invariant known as the 
\textit{basic counting invariant}.

For a finite rack $X$, the rack rank $N$ is always finite -- indeed, $N$ is
the exponent or order of the kink map $\pi:X\to X$ considered as an element 
of the
symmetric group $S_{|X|}$. The finiteness of $N$ for a rack $X$ implies that
the basic counting invariants are periodic in the writhe $w$ of each component
of $L$ with period $N$ -- in particular, the basic counting invariant is
preserved by \textit{$N$-phone cord moves}:
\[\includegraphics{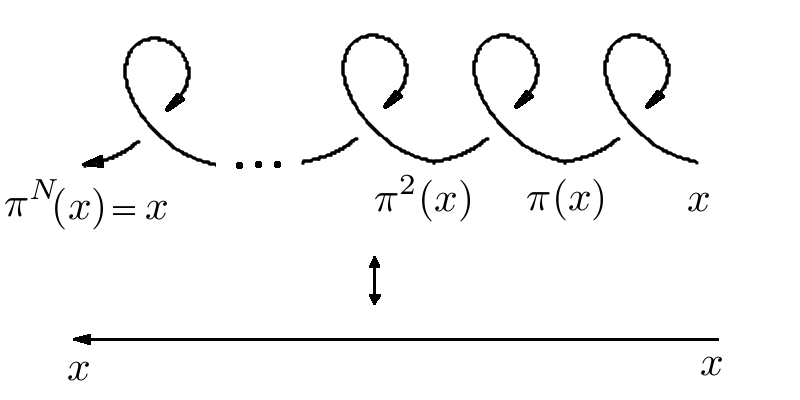}\]

Thus, if we let $W=(\mathbb{Z}_N)^c$ where $c$ is the number of components
of $L$ and let $D(L,\mb{w})$ be a diagram of $L$ with framing vector
$\mb{w}$ for a fixed ordering of the components of $L$, then the number
\[\Phi^{\mathbb{Z}}_{X}(L)=\sum_{\mb{w}\in W} |\mathrm{Hom}(FR(D(L,\mb{w})),X)|\]
is an invariant of the unframed link $L$ called the \textit{integral rack
counting invariant}. In the special case that $X$ is a
quandle, i.e. $N=1$, this is just the basic counting invariant
$|\mathrm{Hom}(FR(L),X)|$.

\begin{example}\label{ex1}
\textup{The constant action rack $X$ with $\sigma=(12)$ has rack rank 2. We can
interpret the rack operation as a labeling rule which says that each time
an arc goes under a crossing, the label switches from $1$ to $2$ or from
$2$ to $1$. Since we have $N=2$, to compute $\Phi^{\mathbb{Z}}_X(L)$ we need
a complete set of framing vectors over $(\mathbb{Z}_2)^c$. For example, the
Hopf link $H_2$ and the 2-component unlink $U_2$ both have four labelings 
by $X$, but they
occur in different framings, with the only valid labelings of the unlink
occurring in framing $(0,0)\in(\mathbb{Z}_2)^2$ and those of the Hopf link 
occurring in framing $(1,1)\in(\mathbb{Z}_2)^2$. }
\[\includegraphics{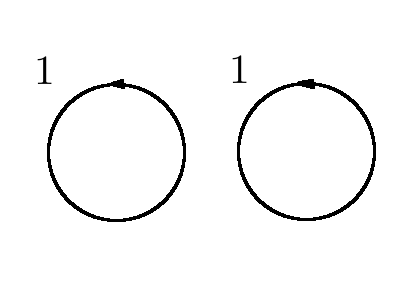} \quad
\includegraphics{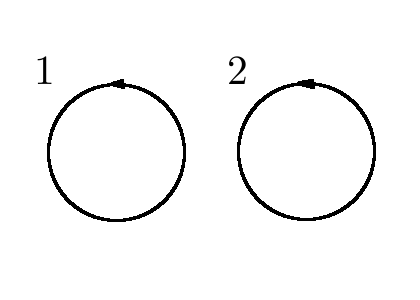} \quad
\includegraphics{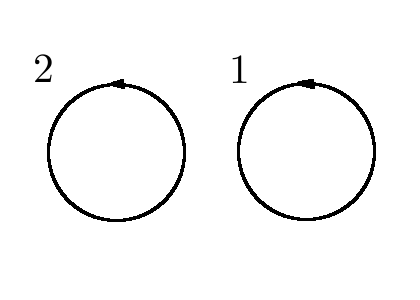} \quad
\includegraphics{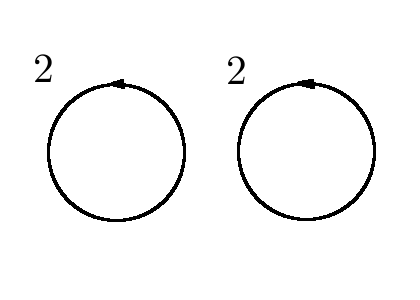} \quad
\]
\[\includegraphics{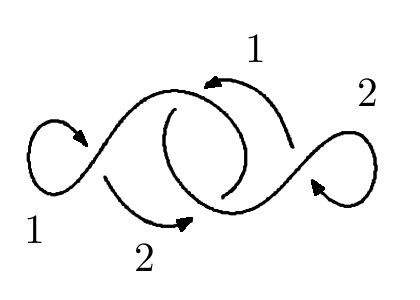} \quad
\includegraphics{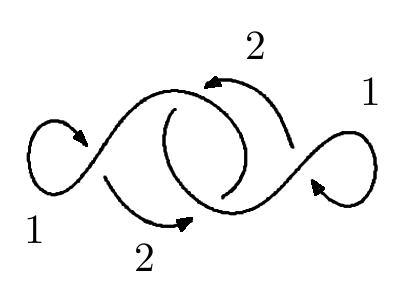} \quad
\includegraphics{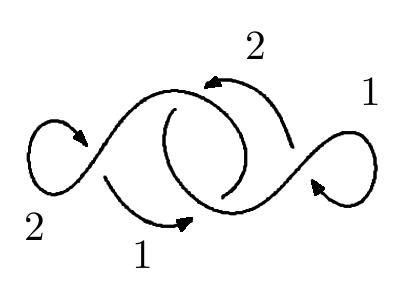} \quad
\includegraphics{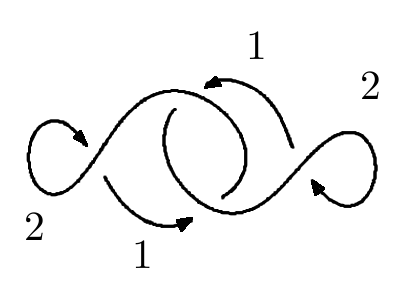} \quad
\]
\end{example}

In example \ref{ex1}, the integral counting invariants defined by the 
given rack $X$ do not distinguish the two links, but we can define an 
\textit{enhancement}, i.e. a stronger invariant with the original invariant
as a specialization, which does (see\cite{N}):

\begin{definition}
\textup{Let $X$ be a rack with rack rank $N$, $L$ a link of $c$ components,
$W=(\mathbb{Z}_N)^c$ and for $\mb{w}=(w_1,\dots, w_c)\in W$ let 
$q^{\mb{w}}=q_1^{w_1}\dots q_c^{w_c}$. Then the \textit{writhe-enhanced rack 
counting invariant} of $L$ defined by $X$ is}
\[\phi^W_X(L)=\sum_{\mb{w}\in W}|\mathrm{Hom}(FR(D(L,\mb{w})),X)|q^{\mb{w}}. \]
\end{definition}

\begin{example}
\textup{In example \ref{ex1}, we had 
$\Phi^{\mathbb{Z}}_X(U_2)=4=\Phi^{\mathbb{Z}}_X(H_2)$; the writhe-enhanced 
invariant detects the difference, with 
$\Phi^{W}_X(U_2)=4\ne4q_1q_2=\Phi^{W}_X(H_2)$.}
\end{example}

\section{\large\textbf{$(t,s)$-racks}}
\label{sec3}

We will now focus on a particular type of rack described in
\cite{FR}.
Let $\ddot\Lambda=\mathbb{Z}[t^{\pm 1},s]/(s^2-(1-t)s)$ and similarly
let $\ddot\Lambda_n=\mathbb{Z}[t^{\pm 1},s]/(n,s^2-(1-t)s)
=\mathbb{Z}_n[t^{\pm 1},s]/(s^2-(1-t)s)$.

\begin{definition}
\textup{Let $X$ be a module over $\ddot\Lambda$. Then $X$ is a rack with
\[x\tr y=tx+sy,\quad \mathrm{and}\quad x\tr^{-1}y = t^{-1}(x-sy)\]
known as a \textit{$(t,s)$-rack} \cite{FR}. If $s=1-t$ then $X$ is a 
quandle known as an \textit{Alexander quandle}.}
\end{definition} 

\begin{lemma}
\textup{If $X$ is a $(t,s)$-rack, then $\pi(x)=(t+s)x$.}
\end{lemma}

\begin{proof} Let $X$ be a $(t,s)$-rack. Then for any $x\in X$, we have
\[\pi(x)=x\tr x=tx+sx=(t+s)x.\]
\end{proof}

\begin{corollary}
For any $(t,s)$-rack $X$, the rack rank $N(X)$ is the minimal integer $N\ge 1$
such that $(t+s)^Nx=x$ for all $x\in X$.
\end{corollary}

Let $X=R$ for a commutative ring $R$. We can make $X$ a $(t,s)$-rack by
selecting an invertible $t\in R$ and an element $s\in R$ satisfying 
$s^2=(1-t)s$. If $R$ is finite, e.g. $R=\mathbb{Z}_n$, then $X$ is a
finite rack. Racks of this type with $R=\mathbb{Z}_n$ will be called
\textit{linear $(t,s)$-racks}, since we have $R=\ddot\Lambda_n/(t-a,s-b)$
for some $t=a,s=b \in \mathbb{Z}_n$. If $R$ is a field, then either $s=0$
and we have a constant action rack with $\sigma(x)=tx$, or $s$ is 
invertible; if $s$ is invertible, then $s^2=(1-t)s$ implies $s=1-t$ and 
our rack is a quandle. Thus, we have:

\begin{proposition}
Every linear $(t,s)$-rack $X=\mathbb{Z}_p$ for $p$ prime is
either a constant action rack or a linear Alexander quandle.
\end{proposition}

\begin{example}\label{num4}
\textup{The smallest nonquandle example of a linear $(t,s)$-rack is 
$X=\mathbb{Z}_4=\{1,2,3,4\}$ with $t=1$ and $s=2$. Then we have 
$s^2=2^2=4=0$ and 
$(1-t)s=(1-1)2=0$. Since the kink map here is $\pi(x)=(s+t)x=3x$ we have 
rack rank $N=2$ and $X$ is a non-quandle rack. The rack matrix of this rack
is}
\[M_{X}=\left[\begin{array}{cccc}
3 & 1 & 3 & 1 \\
4 & 2 & 4 & 2 \\
1 & 3 & 1 & 3 \\
2 & 4 & 2 & 4 \\
\end{array}\right].\]
\end{example}

Another way to get finite $(t,s)$-racks is to take quotients of
$\ddot\Lambda$. The relation $s^2=(1-t)s$ says we can replace any power of
$s$ greater than 1 with an equivalent expression which is linear in $s$; thus
as an abelian group, we have
$\ddot\Lambda\cong \mathbb{Z}[t^{\pm 1}]\oplus \mathbb{Z}[t^{\pm 1}]$.
Then we can get finite $(t,s)$-racks by taking $\ddot\Lambda_n/(p(t))$
for a monic polynomial $p(t)$. 

\begin{example}\label{lin4}
\textup{Let $Y=\ddot\Lambda_2/(t+1)$. The elements of $Y$ 
include $0,1,s$ and $1+s$, and $X$ has operation table}
\[\begin{array}{c|cccc}
\tr & 0   & 1   & s   & 1+s \\ \hline
0   & 0   & s   & 0   & s   \\  
1   & 1   & 1+s & 1   & 1+s \\ 
s   & s   & 0   & s   & 0   \\
1+s & 1+s & 1   & 1+s & 1   \\
\end{array}\]
\textup{Thus, we appear to have a second example of a non-quandle 
$(t,s)$-rack of four elements. However, it is easy to check that
this rack is isomorphic to the 4-element linear $(t,s)$-rack in 
example \ref{num4}, via for example $\phi:Y\to X$ given by 
$\phi(0)=4$, $\phi(1)=1$, $\phi(s)=2$ and $\phi(1+s)=3$. $X$ and $Y$ are
not isomorphic as $\ddot\Lambda$-modules, however, since their additive
structures are different -- as abelian groups, $X=\mathbb{Z}_4$ while 
$Y=\mathbb{Z}_2\oplus\mathbb{Z}_2$.}
\end{example}

We can define a $(t,s)$-rack structure on any abelian group $A$ by selecting 
an automorphism $t:A\to A$ and an endomorphism $s:A\to A$ satisfying the 
conditions that $st=ts$ and that $s^2=(\mathrm{Id}-t)s$. 

\begin{example}\textup{The linear $(t,s)$-rack in example \ref{lin4} can
be expressed as $X=\mathbb{Z}_2\oplus\mathbb{Z}_2$ with $t=\mathrm{Id}$
and $s(x,y)=(0,x)$, while the linear $(t,s)$-rack in example \ref{num4}
has $X=\mathbb{Z}_4$ with $t(x)=x$ and $s(x)=2x$.}
\end{example}

If $\phi:X\to Y$ is an isomorphism of $\ddot\Lambda$-modules, then
$\phi$ is also an isomorphism of $(t,s)$-racks; however, it is clear from 
examples \ref{num4} and \ref{lin4} that rack isomorphism
type does not determine $\ddot\Lambda$-module structure. What conditions
on $\ddot\Lambda$-modules result in isomorphic $(t,s)$-racks?
In \cite{AG,N1} we have a theorem about Alexander quandles, namely:

\begin{theorem}\label{alexq}
Two finite Alexander quandles $M$ and $M'$ are isomorphic as quandles
iff 
\begin{list}{}{}
\item[(i)]{$|M|=|M'|$ and} 
\item[(ii)]{There exists a $\mathbb{Z}[t^{\pm 1}]$-module isomorphism
$h:(1-t)M\to(1-t)M'$.}
\end{list}
\end{theorem}

More colloquially, theorem 1 says that two Alexander quandles of the 
same finite cardinality are isomorphic iff their \textit{$(1-t)$-submodules}
$(1-t)M$ and $(1-t)M'$ are isomorphic as $\mathbb{Z}[t^{\pm 1}]$-modules.
We would like to generalize this result to $(t,s)$-racks.
We first note that the straightforward generalization obtained by simply 
replacing $1-t$ with $s$ does not work; $X=\ddot\Lambda_4/(t-1,s-2)$ and 
$Y=\ddot\Lambda_4/(t-3,s-2)$ both have $s$-submodules $sX$ and $sY$ 
isomorphic to $\ddot\Lambda_2/(t-1,s-0)$ and $|X|=|Y|$, but $Y$ is a 
quandle while $X$ is a rack with rack rank 2.

As in the case of Alexander biquandles in \cite{LN}, we are able 
to give necessary and sufficient conditions for two $(t,s)$-racks to be 
isomorphic. We first need a few lemmas:

\begin{lemma}
If $\phi:X\to Y$ is a homomorphism of $(t,s)$-racks, then
$\phi((t+s)x)=(t+s)\phi(x)$ for all $x\in X$.
\end{lemma}

\begin{proof}
\[\phi((t+s)x)=\phi(x\tr x)=\phi(x)\tr\phi(x)=(t+s)\phi(x).\]
\end{proof}

\begin{lemma} \label{qls}
Let $X$ be a $(t,s)$-rack and let $z\in X$. The bijective map $p_z:X\to X$ 
defined by $p_z(x)=x+z$ is a rack isomorphism if and only if $\pi(z)=z$.
\end{lemma}

\begin{proof} Let $X$ be a $(t,x)$-rack. Then for any $x,y,z\in X$ we have
\[p_z(x\tr y)=p_z(tx+sy)=tx+sy+z\] while
\[p_z(x)\tr p_z(y)=(x+z)\tr (y+z)=tx+tz+sy+sz=tx+sy+(t+s)z.\]
Then $p_z(x\tr y)=p_z(x)\tr p_z(y)$ iff $z=(t+s)z=\pi(z)$. 
\end{proof}

Let $X$ be a $(t,s)$-rack and $A\subset X$ a subset. The 
\textit{$(t+s)$-orbit} of $A$, denoted $\mathcal{O}_{(t+s)}(A)$, is
the set 
\[\mathcal{O}_{(t+s)}(A)=\{(t+s)^k\alpha \ |\ \alpha\in A, 
k\in\mathbb{Z}\}.\]
We will be interested in the case where $A$ is a set of coset representatives
of $X/sX$; note that in such a case multiple elements of 
$\mathcal{O}_{(t+s)}(A)$ may belong to the same coset of $X/sX$.
Moreover, note that since $(t+s)$ is invertible, every element 
$x\in X$ can be written as $x=(t+s)y$ for some $y=\alpha+\omega$ with 
$\alpha\in A$, $\omega\in sX$; then we have 
\[x=(t+s)\alpha+(t+s)\omega=(t+s)\alpha+\omega'\]
where $\alpha\in A$, $\omega'\in sX$. In particular, every element of 
$\mathcal{O}_{(t+s)}(A)$ can be written as $(t+s)\alpha+\omega$ for some 
$\alpha\in A, \omega\in sX$.

\begin{theorem}\label{main}
Two $(t,s)$-racks $X,Y$ are isomorphic if and only if 
\begin{list}{}{}
\item[(i)]{There is an isomorphism of $\ddot\Lambda$-submodules 
$h:sX\to sY$ and}
\item[(ii)]{There are sets of coset representatives $A, B$ for 
$X/sX$ and $Y/sY$ and a bijection 
\[g:\mathcal{O}_{(t+s)}(A)\to \mathcal{O}_{(t+s)}(B)\] such that  
\[ h(s\alpha)=sg(\alpha)\] for all $\alpha\in A$ and 
\[g((t+s)\alpha+\omega)=(t+s)g(\alpha)+h(\omega)\]
for all $(t+s)\alpha+\omega\in\mathcal{O}_{(t+s)}(A)$  with
$\alpha\in A$ and $\omega\in sX$.}
\end{list}
\end{theorem}

\begin{proof} 

($\Rightarrow$) Let $\phi:X\to Y$ be an isomorphism of $(t,s)$-racks.
In $X$ we have $0\tr 0=t0+s0=0+0=0$ so $N(0)=1$, and lemma 
\ref{lem:qord} implies $N(\phi(0))=1$. Then by lemma \ref{qls}
we may assume without loss of generality that $\phi(0)=0$ since if not,
we can replace $\phi$ with $p_{-\phi(0)}\circ\phi$.

Since $\phi$ is a $(t,s)$-rack homomorphism we have
\[\phi(tx+sy)=\phi(x\tr y)=\phi(x)\tr\phi(y)=t\phi(x)+s\phi(y)\]
and since $\phi(0)=0$ we have
\[\phi(tx)=\phi(tx+s0)=t\phi(x)+s\phi(0)=tx+s0=tx\]
and
\[\phi(sy)=\phi(t0+sy)=t\phi(0)+s\phi(y)=t0+sy=sy.\]
Since $t$ is invertible, every element $x\in X$ is $t(t^{-1}x)$. 
Then we have
\begin{eqnarray*}
\phi(sx+sy) & = & \phi(t(t^{-1}sx)+sy) \\
 & = & t\phi(t^{-1}sx)+s\phi(y) \\
 & = & \phi(tt^{-1}sx)+\phi(sy) \\
 & = & \phi(sx)+\phi(sy). \\
\end{eqnarray*}
Not 
every $x\in X$ need satisfy $x=sz$ for some $z\in X$, but for those that do, 
i.e. for the submodule $sX$, we have $\phi$ preserving multiplication 
by both $t$ and $s$ and preserving addition, so the restriction 
$h=\phi|_{sX}$ is an isomorphism of $\ddot\Lambda$-modules. 

Now, let $A$ 
be any set of coset representatives of
$X/sX$. Define $g=\phi|_{\mathcal{O}_{(t+s)}(A)}$ and 
set $B=\{\phi(\alpha)\ |\ \alpha\in A\}$. 
Then for each $\alpha\in A$ we have 
\[h(s\alpha)=\phi(s\alpha)=s\phi(\alpha)=sg(\alpha)\]
and for any $(t+s)\alpha+\omega\in \mathcal{O}_{(t+s)}(A)$ with $\omega=s\gamma$
we have
\begin{eqnarray*}
g((t+s)\alpha +\omega)
& = &\phi((t+s)\alpha+\omega)=\phi(tt^{-1}(t+s)\alpha+s\gamma) \\
& = & t\phi(t^{-1}(t+s)\alpha)+s\phi(\gamma) \\
& = & \phi(tt^{-1}(t+s)\alpha)+\phi(s\gamma) \\
& = & (t+s)\phi(\alpha) + \phi(\omega) \\
& = & (t+s)g(\alpha) +h(\omega)
\end{eqnarray*}
as required.

Finally, note that $B$ is a set of coset representatives for $Y/sY$ since 
if $\beta-\beta'\in sY$ for any $\beta,\beta'\in B$ then 
$\beta=\beta'+s\gamma$ and we have
\[\phi^{-1}(\beta)
=\phi^{-1}(tt^{-1}\beta'+s\gamma)
=t\phi^{-1}(t^{-1}\beta')+s\phi^{-1}(\gamma)
=\phi^{-1}(tt^{-1}\beta)+\phi^{-1}(s\gamma)
=\phi^{-1}(\beta')+\phi^{-1}(s\gamma)
\]
and the corresponding $\alpha=\phi^{-1}\beta$, $\alpha'=\phi^{-1}(\beta')$
satisfy $\alpha-\alpha'\in sX$.

($\Leftarrow$) Let $X$ and $Y$ be $(t,s)$-racks, $h:sX\to sY$ an isomorphism
of $\ddot\Lambda$-modules, and suppose 
$A\subset X$ and $B\subset Y$ are sets of coset
representatives of $X/sX$ and $Y/sY$ respectively, with a bijection
$g:\mathcal{O}_{(t+s)}(A)\to \mathcal{O}_{(t+s)}(B)$ satisfying
\[ h(s\alpha)=sg(\alpha)\] for all $\alpha\in A$ and
\[g((t+s)\alpha+\omega)=(t+s)g(\alpha)+h(\omega)\]
for all $(t+s)\alpha+\omega\in\mathcal{O}_{(t+s)}(A)$. In particular, $\omega=0$
says $g((t+s)^k\alpha)=(t+s)^kg(\alpha)$.
Define $\phi:X\to Y$ by 
\[\phi(\alpha+\omega)=g(\alpha)+h(\omega)\]
where $\alpha\in A$ and $\omega\in sX$. 

To see that $\phi$ is well-defined, suppose $\alpha+\omega=\alpha'+\omega'$
where $\alpha,\alpha'\in\mathcal{O}_{(t+s)}(A)$ and $\omega,\omega'\in sX$.
Then $\alpha'=\alpha+(\omega-\omega')$ and we have
\begin{eqnarray*}
\phi(\alpha'+\omega') 
& = & g(\alpha')+h(\omega')\\
& = & g(\alpha +(\omega-\omega'))+h(\omega') \\
& = & g((t+s)(t+s)^{-1}\alpha +(\omega-\omega')) +h(\omega')\\
& = & (t+s)g((t+s)^{-1}\alpha) +h(\omega-\omega')+h(\omega') \\
& = & (t+s)g((t+s)^{-1}\alpha) +h(\omega)-h(\omega')+h(\omega') \\
& = & (t+s)g((t+s)^{-1}\alpha) +h(\omega) \\
& = & g(\alpha) +h(\omega) \\
& = & \phi(\alpha+\omega)
\end{eqnarray*}
To see that $\phi$ is bijective, note that we can define $\phi^{-1}:Y\to X$
by $\phi^{-1}(\beta+\gamma)=g^{-1}(\beta)+h^{-1}(\gamma)$.

To see that $\phi$ is a homomorphism of $(t,s)$-racks, let $x=\alpha+\omega$ 
and $y=\alpha'+\omega'$ with $\alpha,\alpha'\in A$ and $\omega,\omega'\in sX$, 
and note that $t\alpha=(t+s)\alpha-s\alpha$. Then

\begin{eqnarray*}
\phi(x\tr y) & = & \phi(t(\alpha+\omega) +s(\alpha'+\omega')) \\
& = & \phi((t+s)\alpha -s\alpha +t\omega +s\alpha'+s\omega') \\
& = & g((t+s)\alpha) + h(-s\alpha +t\omega+s\alpha'+s\omega') \\
& = & (t+s)g(\alpha) -h(s\alpha)+ th(\omega)+h(s\alpha')+sh(\omega') \\
& = & tg(\alpha)+sg(\alpha)-h(s\alpha)+th(\omega)+h(s\alpha')+sh(\omega') \\
& = & tg(\alpha)+h(s\alpha)-h(s\alpha)+th(\omega)+sg(\alpha')+sh(\omega') \\
& = & tg(\alpha)+th(\omega)+sg(\alpha')+sh(\omega') \\
& = & t(g(\alpha)+ h(\omega))+s(g(\alpha)+h(\omega')) \\
& = & \phi(x)\tr \phi(y) \\
\end{eqnarray*}
as required. 

\end{proof}

\begin{remark}
\textup{Note that if $s=1-t$ and $X$ is an Alexander quandle, then 
$s+t=1$ and $\mathcal{O}_{s+t}(A)$ is just $A$. Then the fact that $A$ is
a set of coset representatives means that condition $(ii)$ reduces to 
the requirement that $h(s\alpha)=sg(\alpha)$, and it is shown in \cite{N}
that this condition can always be satisfied.}
\end{remark}

We end this section with a few interesting observations about $(t,s)$-racks.

\medskip

In every rack, the $\tr$ and $\tr^{-1}$ operations are 
right-distributive\footnote{At least, when we write the quandle operation
as a right action following Joyce. In some works such as \cite{AG} the
rack operations are written as left actions, in which case the rack axioms
require left-distributivity.}; if a quandle is Alexander, however,
the quandle operations are also left-distributive. This property does
not extend to more general $(t,s)$-racks:

\begin{proposition}
A $(t,s)$-rack $X$ is left-distributive if and only if $X$ is an Alexander
quandle.
\end{proposition}

\begin{proof}
Let $X$ be a $(t,s)$-rack. Then
\[x\tr(y\tr z)=tx+s(ty+sz) =tx+tsy+s^2z\]
while \[(x\tr y)\tr (x\tr z)=t(tx+sy)+s(tx+sz)=(t^2+st)x+tsy+s^2z.\]
Then $x\tr(y\tr z)=(x\tr y)\tr (x\tr z)$ if and only if $(ts+t^2)x=tx$ for
all $x\in X$, i.e., iff $t^2x=t(1-s)x$, which implies $tx=(1-s)x$ and 
hence $sx=(1-t)x$. We then have a $\mathbb{Z}[t^{\pm 1}]$-module structure
on $X$ induced by taking the quotient of $\ddot\Lambda$ by the ideal 
generated by $s-(1-t)$, and the $(t,s)$-rack operation on $X$ becomes
\[x\tr y=tx+sy=tx+(1-t)y\] 
and $X$ is an Alexander quandle. 
\end{proof}

Our next observation notes that $(t,s)$-racks contain Alexander quandles
not just in the categorical sense, but literally:

\begin{proposition}
Let $X$ be a rack. The subset $Q(X)\subset X$ of all elements of
$X$ of rack rank $N=1$ is a quandle, known as the \textit{maximal subquandle}
of $X.$ If $X$ is a $(t,s)$-rack, then $Q(X)$ is an Alexander quandle.
\end{proposition}

\begin{proof} To see that $Q(X)=\{x\in X\ |\ x\tr x=x\}$ is a subrack, note
that $x,y\in Q$ implies 
\[x\tr y=(x\tr x)\tr y =(x\tr y)\tr (x\tr y)\]
and $x\tr y\in Q$. Then $Q(X)$ is a rack with rack rank $N=1$, so $Q(X)$ is
a quandle.

Now let $X$ be a $(t,s)$-rack. To see that $Q$ is Alexander, note that if 
$x\in Q$ then $(t+s)x=x$ and we have $sx=(1-t)x$. Then for any $x,y\in Q$, 
we have
\[x\tr y=tx+sy=tx+(1-t)y\]
and $Q$ is an Alexander quandle.
\end{proof}

For general racks $Q(X)$ may be empty (none of the quandle axioms
are existentially quantified, so the empty set satisfies the quandle axioms
vacuously), but for $(t,s)$-racks $Q(X)$ always contains at least $0$. 
Indeed, we have

\begin{corollary}\label{q:x}
For any $(t,s)$-rack $X$, $sX$ is a subquandle of $Q(X)$.
\end{corollary}

\begin{proof}
To see that $sX$ is closed under $\tr$, note that 
\[sx\tr sy = tsx+s^2y=s(tx+sy)\in sX.\] To see that $sX\subset Q(X)$,
let $x=sx'$; then we have
\[sx'\tr sx'=tsx'+s^2x'=(ts+s^2)x'=sx'\]
and $x\in Q(X)$.
\end{proof}

We note that  $sX$ may be a proper subquandle of $Q(X)$: take for instance
$X=\mathbb{Z}_4$ with $t=3$ and $s=2$; then $t+s=1$ and $Q(X)=X$, but 
$sX=\{0,2\}\subsetneq X$.

\section{\large\textbf{Enhanced link invariants}}\label{sec4}

In this section we define a few enhancements of the rack counting invariant
$\Phi_X^{\mathbb{Z}}$ when $X$ is a $(t,s)$-rack.

For our first enhancement, we note that a $(t,s)$-rack is not 
just a rack, but also has the structure of a $\ddot\Lambda$-module. We can 
use this extra structure to define enhancements of the rack counting invariant. 
Let $T$ be a finite $(t,s)$-rack with rack rank $N$ and let $A_T$ be $T$ considered
as an abelian group. For any subset $S\subset T$, let $AC(S)$ be the additive 
closure of $S$, i.e. the subgroup of $A_T$ generated by $S$. For each 
homomorphism $f:FR(L,\mathbf{w})\to T$ we can use the additive closure of 
the image subrack of $f$, as a signature of $f$
to obtain an enhancement of $\Phi_X^{\mathbb{Z}}$.

\begin{definition}
\textup{Let $X$ be a $(t,s)$-rack and $L=L_1\cup \dots \cup L_c$ an
oriented link of $c$ ordered components. The \textit{additive $(t,s)$-rack
enhanced multiset} of $L$ with respect to $T$ is the multiset of abelian 
groups}
\[\Phi^{ts,+M}_X(L)= 
\left\{AC(\mathrm{Im}(f)) \ | \ f\in\mathrm{Hom}(FR(L,\mathbf{w}),X),\ 
\mathbf{w}\in W\right\}.\]
\textup{For ease of comparison, we also define the \textit{additive 
$(t,s)$-rack enhanced polynomial} of $L$ with respect to $X$ to be}
\[\Phi^{ts,+}_X(L)
=\sum_{\mathbf{w}\in W} \left(
\sum_{f\in\mathrm{Hom}(FR(L,\mathbf{w}),X)}u^{|AC(\mathrm{Im}(f))|}\right)\]
\textup{where $W=(\mathbb{Z}_N)^c$.} 
\end{definition}

Note that for any $f\in\mathrm{Hom}(FR(L,\mathbf{w}),X)$, 
$X$-labeled blackboard-framed Reidemeister moves and $N$-phone cord
moves do not change the image subrack $\mathrm{Im}(f)$, and thus the above
quantities are link invariants. It is clear that the rack counting invariant 
can be obtained from either form of the enhanced invariant by taking the 
cardinality in the multiset case or by evaluating $u=1$ in the polynomial 
case. We also note that the multiset form of the invariant is stronger
than the polynomial form since the polynomial form forgets the abelian
group structure of the signatures, keeping only their cardinalities.

In the proof of proposition \ref{2n} we illustrate a method for computing 
$\Phi^{ts,+}_X$ by computing $\Phi^{ts,+}_X(L)$ for all $(2,n)$-torus links
for a choice of $(t,s)$-rack $X$.

\begin{proposition}\label{2n}
\textup{Let $X=\mathbb{Z}_4$ with $(t,s)$-rack operation $x\tr y=x+2y$. Then
the $(2,n)$ torus link $T_{(2,n)}$ has $\Phi_{X}^{ts,+}$ values given by}
\[
\Phi_{X}^{ts,+}(T_{(2,n)})=\left\{
\begin{array}{ll}
4u+12u^2+20u^4, & n\equiv 0 \ \mathrm{mod}\ 4, \\
2u+2u^2+2u^4, & n\equiv 1,3 \ \mathrm{mod}\ 4, \\
4u+12u^2+4u^4, & n\equiv 2 \ \mathrm{mod}\ 4. \\
\end{array}
\right.
\]
\end{proposition}

\begin{proof}
Let $X$ be the $(t,s)$-rack $X=\mathbb{Z}_4$ with $t=1$ and $s=2$.
Recall that the $(2,n)$ torus link $T_{(2,n)}$ is the closure of the 2-strand 
braid with $n$ positive twists. We first note that the set of rack labelings
of $T_{(2,n)}$ is periodic with period 4:
\[\includegraphics{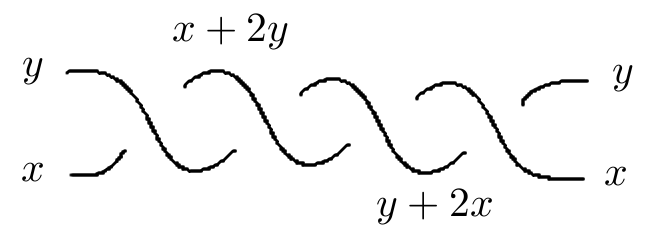}.\]

There are four cases; we will show the first two. Closing the braid
gives us a system of equations in $\mathbb{Z}_4$; each solution to this
system determines a valid rack labeling. 

When 
$n\equiv 0\ \mathrm{mod}\ 4$, $T_{(2,n)}$ is a two-component link and we 
need to stabilize once on each component (the braid version of the 
Reidemeister I move) in order to get a complete set of
writhes mod $2$:
\[\begin{array}{|lllll|} \hline
\mathrm{Writhe\ vector} &
\mathrm{Diagram} & \mathrm{System} & \mathrm{Reduces\ to} 
& \mathrm{Contribution} \\ \hline
(0,0) &
\raisebox{-0.25in}{\includegraphics{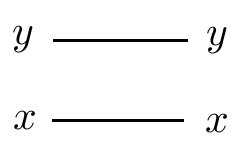}} & 
\begin{array}{rcl}
y & = & y \\
x & = & x \\ 
\end{array}
 &
\begin{array}{rcl}
y & = & y \\
x & = & x \\ 
\end{array}
&
u+3u^2+12u^4 \\ \hline
(0,1) &
\raisebox{-0.5in}{\includegraphics{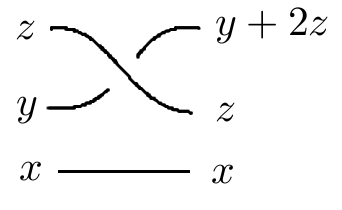}} & 
\begin{array}{rcl}
z & = & y+2z \\
y & = & z \\
x & = & x 
\end{array}
&
\begin{array}{rcl}
y & = & 3y \\
x & = & x \\ 
\end{array}
& u+3u^2+4u^4 \\ \hline
(1,1) &
\raisebox{-0.5in}{\includegraphics{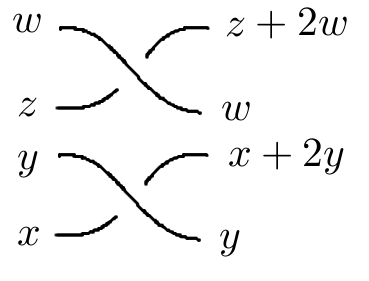}} & 
\begin{array}{rcl}
w & = & z+2w \\
z & = & w \\
y & = & x+2y \\
x & = & y 
\end{array}
&
\begin{array}{rcl}
z & = & 3z \\
x & = & 3x \\ 
\end{array}
& u+3u^2 \\ \hline
\end{array}
\]
The writhe vector $(1,0)$ has the same contribution as
the $(0,1)$ writhe vector due to the symmetry of the link.
Thus, we have $\Phi^{ts,+}_{X}(T_{(2,n)})=4u+12u^2+20u^4$ for 
$n\equiv 0 \ \mathrm{mod}\ 4.$

When 
$n\equiv 1\ \mathrm{mod}\ 4$, $T_{(2,n)}$ is a knot and we need to stabilize
once:
\[\begin{array}{|lllll|} \hline
\mathrm{Writhe} &
\mathrm{Diagram} & \mathrm{System} & \mathrm{Reduces\ to} 
& \mathrm{Contribution} \\ \hline
0 &
\raisebox{-0.25in}{\includegraphics{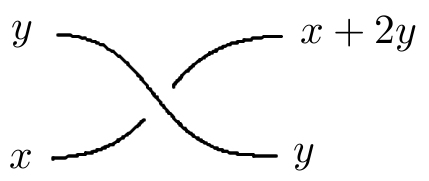}} & 
\begin{array}{rcl}
y & = & x+2y \\
x & = & y \\ 
\end{array}
 &
\begin{array}{rcl}
2x & = & 0 \\
x & = & y \\ 
\end{array}
&
u+u^2 \\ \hline
1 &
\raisebox{-0.5in}{\includegraphics{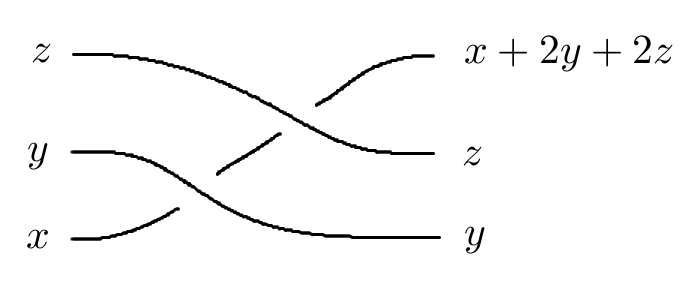}} & 
\begin{array}{rcl}
z & = & x+2y+2z \\
y & = & z \\
x & = & y 
\end{array}
&
x=y=z
& u+u^2+2u^4 \\ \hline
\end{array}
\]
Thus, we have $\Phi^{ts,+}_{X}(T_{(2,n)})=2u+2u^2+2u^4$ for 
$n\equiv 1 \ \mathrm{mod}\ 4.$

Similar computations give us the $n\equiv 2,3 \ \mathrm{mod\ 4}$ cases.
\end{proof}

\begin{remark}
\textup{Repeating the computation in proposition \ref{2n} with the
$(t,s)$-rack structure $X=\ddot\Lambda_2/(t^2+1)$, we get for example}
\[\Phi_{X}^{ts,+}(T_{(2,n)})=4u+22u^2+10u^4\]
\textup{when $n\equiv 0\ \mathrm{mod}\ 4$. This result shows that different
extra structures on the same rack can indeed define different enhanced
invariants, as suggested in \cite{NN}. While this may seem counter-intuitive,
it merely reflects the fact that different extra structures on a rack
(or quandle, biquandle, etc.) differ in which labelings get assigned 
different signatures. In particular, to define $\Phi^{ts,+}_{X}$ it is 
not enough to know the rack matrix of $X$; we need the full 
$\ddot\Lambda$-module structure.}
\end{remark}

Our next example gives a table of values of $\Phi^{ts,+}_X(L)$ for all
prime knots with up to eight crossings and all prime links with up to
seven crossings using our \texttt{python} code, available at
\texttt{http://www.esotericka.org}. In particular, these values
demonstrate that $\Phi^{ts,+}_X(L)$ is stronger in general than 
the integral counting invariant $\Phi^{\mathbb{Z}}_X(L)$.

\begin{example}
\textup{Let $X$ be the $(t,s)$-rack $X=\mathbb{Z}_{12}$ with $t=11$ and $s=2$;
in fact, $X$ is an Alexander quandle. We computed the value of 
$\Phi^{ts,+}_{X}$ on all prime knots with up to eight crossings and all prime 
links with up to seven crossings as listed in the Knot Atlas \cite{KA} using
\texttt{python} code available at \texttt{http://www.esotericka.org}. The
results are collected in the table below. Note in particular that while
$\Phi^{ts,+}_X$ determines $\Phi^{\mathbb{Z}}_X$ (since $\Phi^{\mathbb{Z}}_X$
may be obtained from $\Phi^{ts,+}_X$ by setting $u=1$), the values in this 
example demonstrate that $\Phi^{\mathbb{Z}}_X$ does not determine 
$\Phi^{ts,+}_X$ -- for example, both $L6a1$ and $L6a5$ have 
$\Phi_{X}^{\mathbb{Z}}=144$, but are distinguished by $\Phi^{ts,+}_X$ --  
making $\Phi^{ts,+}_{X}$ a strictly stronger invariant.}
\[
\begin{array}{r|l}
\Phi^{ts,+}_X(L) & L \\ \hline
u+u^2+2u^3+2u^4+2u^6+4u^{12} & 4_1,5_1,5_2,6_2,6_3,7_1,7_2,7_3,7_5,7_6,8_1,8_2,8_3,8_4,8_6,8_7,8_8,8_9, \\
 & 8_{12},8_{13},8_{14},8_{16},8_{17}, \\
u+u^2+8u^3+2u^4+8u^6+16u^{12} & 3_1,6_1,7_4,7_7,8_5,8_{10},8_{11},8_{15}, 8_{19}, 8_{20},8_{21} \\
u+u^2+26u^3+2u^4+26u^6+52u^{12} & 8_{18} \\
u+3u^2+2u^3+4u^4+6u^6+8u^{12} & L2a1, L6a2, L7a6 \\
u+3u^2+2u^3+12u^4+6u^6+24u^{12} & L4a1,L5a1,L7a2,L7a3,L7a4,L7n1,L7n2 \\
u+7u^2+2u^3+8u^4+14u^6+16u^{12} & L6n1,L7a7 \\
u+3u^2+8u^3+4u^4+24u^6+32u^{12} & L6a3, L7a5 \\
u+7u^2+8u^3+8u^4+56u^6+64u^{12} & L6a5 \\
u+3u^2+8u^3+12u^4+24u^6+96u^{12} & L6a1, L7a1 \\
u+7u^2+2u^3+56u^4+14u^6+112u^{12} & L6a4 \\
\end{array}
\]
\end{example}

For our next enhancement, we note that multiplication by $s$ in a 
$(t,s)$-rack is a rack homomorphism which projects $X$ onto the Alexander 
subrack $sX$; if $z=x\tr y=tx+sy$ then we have
\[sx\tr sy = t(sx)+s(sy)=s(tx+sy)=sz\] 
In particular, for any $X$-labeling $f$ of
a link diagram, there is a corresponding $sX$-labeling obtained by
multiplying every label by $s$. Since this corresponding labeling
is preserved by blackboard-framed Reidemeister moves and $N$-phone cord
moves, the projection onto the subrack $s\mathrm{Im}(f)=\mathrm{Im}(s(f))$ 
can be used as a signature of $f$ to define enhancements.
\[
\includegraphics{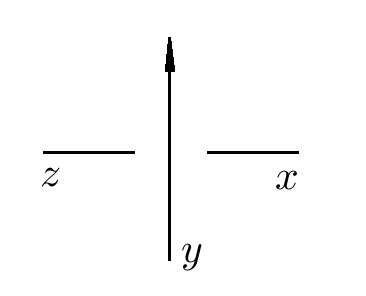} 
\raisebox{0.5in}{$\stackrel{\times s}{\longrightarrow}$} 
\includegraphics{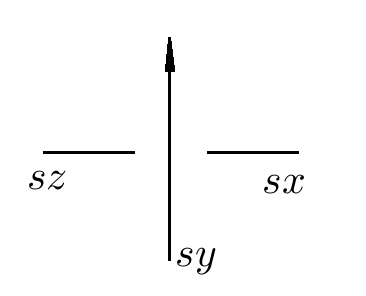} 
\]
Thus we have:

\begin{definition}
\textup{Let $X$ be a $(t,s)$-rack with rack rank $N$, 
$L=L_1\cup \dots \cup L_c$ an oriented link of $c$ ordered components, and
$W=(\mathbb{Z}_N)^c$ the space of writhe vectors mod $N$. For each
rack homomorphism $g:FR(L,\mathbf{w})\to sX$, let 
\[s^{-1}(g)=\{f\in\mathrm{Hom}(FR(L,\mathbf{w}),X)\ :\ g=sf\} \]
be the set of rack labelings of $(L,\mathbf{w})$ by $X$ which project to
$g$ under multiplication by $s$. Then the \textit{$s$-enhanced multiset} 
of $L$ with respect to $X$ is}
\[\Phi^{ts,sM}_{X}(L)=\left\{ s^{-1}(g) : g\in\mathrm{Hom}(FR(L,\mathbf{w}),sX), 
\mathbf{w}\in W\right\}\]
\textup{and the \textit{$s$-enhanced polynomial} of $L$ with respect to $X$ is}
\[\Phi^{ts,s}_{X}(L)=\sum_{\mathbf{w}\in W}\left(\sum_{g\in\mathrm{Hom}(FR(L,X),sX)}
u^{|s^{-1}(g)|}\right).\]
\end{definition}

\begin{remark}
\textup{This enhancement is slightly different from the usual enhancements
in that recovery of $\Phi^{\mathbb{Z}}_X$ requires not evaluating $\Phi^{ts,s}_X$
at $u=1$ but rather summing the product of the coefficient times exponent
for each term in $\Phi^{ts,s}_X$. Note that we can also regard $\Phi^{ts,s}_X$ 
as an enhancement of the quandle counting invariant with respect to the 
Alexander subquandle $sX$; this enhancement is related to the rack module
enhancements defined in \cite{HHNYZ}.}
\end{remark}

\begin{example}
\textup{Taking the 
$(t,s)$-rack 
$X=\mathbb{Z}_4$ with $t=3$ and $s=2$, we computed $\Phi_{X}^{ts,s}$
for all prime links with up to seven crossings as listed in the Knot 
Atlas \cite{KA}. The results are listed in the table below; in particular, 
note that $\Phi^{ts,s}_X$ distinguishes the links $L4a1$ and $L6a5$ which
have the same integral rack counting invariant value $\Phi_{X}^{\mathbb{Z}}=16$.
Since this rack has orbits which are constant action racks, 
the invariant has the same value, $\Phi^{ts,s}_X=2u^2$, for all knots.}
\[\begin{array}{r|l}
\Phi^{ts,s}_X & L \\ \hline
2u+2u^3 & L2a1, L6a2, L6a3, L7a5,L7a6 \\
2u+2u^2+2u^5 & L6a5, L6n1, L7a7 \\
4u+4u^3 & L4a1, L5a1, L6a1, L7a1, L7a2, L7a3, L7a4, L7n1, L7n2 \\
8u+8u^2+8u^5 & L6a4 \\
\end{array}\]
\end{example}

We end this section with an application. In recent works such as \cite{ORS,HS},
a partial ordering on knot types is defined by setting
\[K>K' \iff \exists \phi:\pi_1(S^3\setminus K)\to \pi_1(S^3\setminus K') \]
where $\phi$ is a surjective group homomorphism. Replacing the knot group with
the knot quandle yields a related ordering in which $\phi$ is required to
preserve peripheral structure.

Let us define a partial ordering $\succ$ on $\mathbb{Z}[u]$ by 
\[\sum_{k=0}^n \alpha_ku^k \succ \sum_{k=0}^n \beta_ku^k \iff \alpha_k>\beta_k\]
for all $k=0,\dots,n$. Then we have:

\begin{proposition}
If there exists a surjective homomorphism from the knot quandle
of a knot $K$ onto the knot quandle of $K'$, then 
\[\Phi^{ts,+}_X(K)\succ\Phi^{ts,+}_X(K')\]
for all Alexander quandles $X$.
\end{proposition}

\begin{proof}
For any quandle homomorphism $f:Q(K')\to X$, the map $f\circ\phi:Q(K)\to X$
is a quandle homomorphism. Moreover, 
$\mathrm{Im}(f)\subset\mathrm{Im}(f\circ\phi)$,
since $x\in \mathrm{Im}(f)$ says $x=f(a)$ for some $a\in Q(K')$, and 
surjectivity of $\phi$ then says $a=\phi(b)$ for some $b\in Q(K)$; then we have
$x=f(a)=f(\phi(b))=f\circ\phi(b)$ and $x\in\mathrm{Im}(f\circ\phi)$. Conversely,
if $x\in\mathrm{Im}(f\circ\phi)$ then $x=f\circ\phi(b)$ for some $b\in Q(X)$
and $x=f(\phi(b))$ implies $x\in\mathrm{Im}(f)$. Thus we have
$\mathrm{Im}(f)=\mathrm{Im}(f\circ\phi).$

Then every contribution $u^{|AC(\mathrm{Im}(f))|}$ to $\Phi^{ts,+}_X(K')$ is
matched by an equal contribution to $\Phi^{ts,+}_X(K)$, and we have
\[\Phi^{ts,+}_X(K)\succ\Phi^{ts,+}_X(K')\]
as required.
\end{proof}

This proposition means that for Alexander quandles $X$, $\Phi^{ts,+}_X(K)$ can
provide us with obstructions for knot ordering. Indeed, every finite
Alexander quandle $X$ defines its own partial ordering $>_X$ of knots by 
\[K>_X K' \iff \Phi^{ts,+}_X(K)\succ\Phi^{ts,+}_X(K').\]
For instance, in the quandle ordering defined by the Alexander quandle
$X=\ddot\Lambda_{12}/(t-11,s-2)$ in example 6, we have $4_1<_X3_1<_X8_{18}$, 
etc.

\section{\large\textbf{Questions}}\label{sec5}

In this section we collect questions for future research.

\medskip

Let $X$ be a $(t,s)$-rack. When $(t+s)x=x$ so that $X$ is an Alexander quandle, 
the enhanced invariants defined in section \ref{sec4} are also defined
for knotted surfaces in $\mathbb{R}^4$. We have not looked in any detail
at how effective these enhancements may be at distinguishing knotted surfaces
or what relationship they might have with triple point number, etc. This
might prove to be an interesting direction for future investigation.

The enhanced invariants defined in section \ref{sec4} are also well-defined
without modification for virtual knots and links. It is known that certain
writhe-enhanced rack counting invariant values are impossible for classical
links but possible for virtual links, providing a method of detecting
non-classicality. Does anything similar happen with $(t,s)$-rack enhanced
invariants?

The conditions given in proposition \ref{main} seem unsatisfying; is there
a simpler necessary and sufficient condition which can replace $(ii)$, e.g., 
$X$ and $Y$ have conjugate kink maps in $S_{|X|}$ or equal rack polynomials?

In light of the observations at the end of section 3, a $(t,s)$-rack with 
rack rank $N$ of the form $\ddot\Lambda/I$ for an ideal $I$ can be viewed 
as an extension 
of an Alexander quandle $A=\mathbb{Z}[t^{\pm 1}]/I'$ obtained by adjoining a 
variable $s$ and modding out by $s^2-(1-t)s$, $(s+t)^N-1$ and possibly
additional polynomials. What conditions on these polynomials are
required to yield isomorphic $(t,s)$-racks?

\noindent
\textsc{Department of Mathematical Sciences, \\
Claremont McKenna College, \\
850 Columbia Ave., \\
Claremont, CA 91711}

\end{document}